\documentclass{amsart}
\usepackage{amssymb}
\usepackage{amsfonts}
\usepackage{amsmath}
\usepackage{url}
\usepackage{color}
\usepackage{setspace}
\usepackage{enumerate}

\usepackage{epsfig}
\usepackage{graphicx}
\usepackage{psfrag}
\usepackage{wrapfig}
\newtheorem{theorem}{Theorem}

\newtheorem{proposition}[theorem]{Proposition}
\newtheorem{corollary}[theorem]{Corollary}

\theoremstyle{remark}

\newtheorem{definition}[theorem]{Definition}

\newcommand{\ignore}[1]{}

\begin{document}

\title{On Lorentz spacetimes of constant curvature}
\author{Fran\c{c}ois Gu\'eritaud}
\address{CNRS and Universit\'e Lille 1, Laboratoire Paul Painlev\'e, 59655 Villeneuve d'Ascq Cedex, France}
\email{\newline francois.gueritaud@math.univ-lille1.fr}
\thanks{This work was partially supported by the Agence Nationale de la Recherche under the grants DiscGroup (ANR-11-BS01-013) and ETTT (ANR-09-BLAN-0116-01), and through the Labex CEMPI (ANR-11-LABX-0007-01).}
\date{July 2013}

\begin{abstract}
We describe in parallel the Lorentzian homogeneous spaces $G=\mathrm{PSL}(2,\mathbb{R})$ and $\mathfrak{g}=\mathfrak{psl}(2,\mathbb{R})$, and review some recent results relating the geometry of their quotients by discrete groups.
\end{abstract}

\maketitle

\section{Introduction}

Let $G:=\mathrm{PSL}(2,\mathbb{R})$ be the group of projective $2\times 2$ matrices with positive determinant. Identifying the hyperbolic plane $\mathbb{H}^2$ with the space of complex numbers~$z$ that have positive imaginary part, we can view $G$ as the group of orientation-preserving isometries of $\mathbb{H}^2$, acting by M\"obius transformations
$$\left [ \begin{array}{cc} a & b \\ c & d \end{array} \right ] \cdot z = \frac{az+b}{cz+d}~.$$
Let $\mathfrak{g}:=\mathfrak{psl}(2,\mathbb{R})$ be the Lie algebra of $G$. We can identify $\mathfrak{g}$ with the vector space of traceless matrices (a copy of $\mathbb{R}^3$), and $G$ with a certain open subset of $3$-dimensional projective space $\mathbb{P}^3\mathbb{R}$. Since the determinant of a $2\times 2$ matrix is a quadratic form in its entries, $G$ is bounded in $\mathbb{P}^3\mathbb{R}$ by a quadric $\partial_\infty G$ (the space of rank-1 projective matrices, defined by $ad=bc$). See Figure \ref{GL2}.

\begin{figure}[h!] \centering
%\psfrag{a}{$a$}
\includegraphics[width=12.5cm]{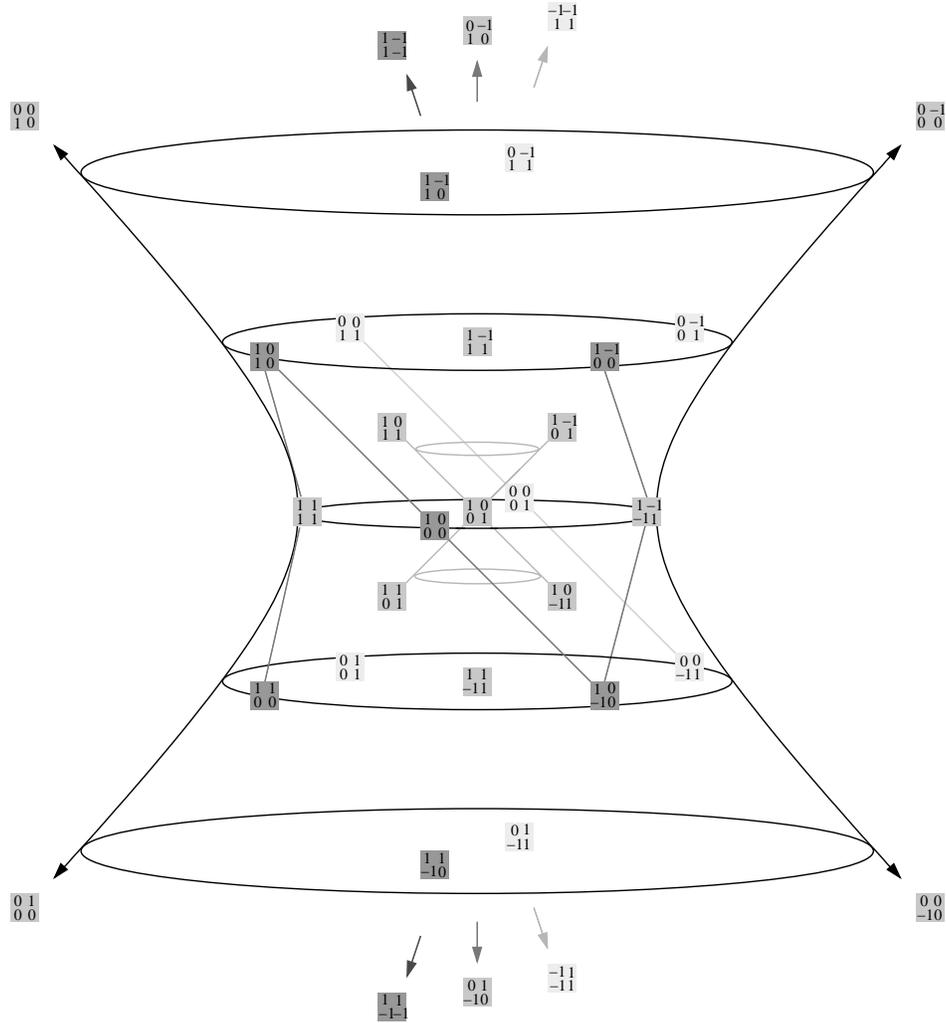}
\caption{A view of $G=\mathrm{PSL}(2,\mathbb{R})$ as the interior of a quadric in $\mathbb{P}^3\mathbb{R}$. We have plotted all matrices that have a representative with entries in $\{1,-1,0\}$. The identity matrix is in the center, and the traceless matrices are at infinity in directions indicated by little arrows. Matrices at infinity are identified in opposite pairs, while for all other matrices we chose the representative with positive trace. Shades of grey are used to help indicate depth (closer is darker) and have no mathematical meaning.} \label{GL2} 
\end{figure}

Let $\mathbf{G}_1:=G\times G$, with multiplication given as usual by $(g,h)(g',h')=(gg', hh')$. 
Let $\mathbf{G}_2:=G\rtimes G$ where the right factor acts on the left factor by conjugation: in other words, the product in $\mathbf{G}_2$ is given by 
\begin{equation}
(\alpha, a)(\beta, b)=(\alpha a \beta a^{-1}, ab)~. \label{semidirectproduct}
\end{equation}
There is a natural isomorphism 
$$\begin{array}{rrcl} \varphi: & \mathbf{G}_2 & \longrightarrow & \mathbf{G}_1 \\ & (\alpha, a) &\longmapsto & (\alpha a, a)~. \end{array}$$

Moreover, $\mathbf{G_2}$ acts on $G$ by the formula
\begin{equation}
(\alpha, a)\cdot_2 x := \alpha a x a^{-1} \label{semidirectaction}
\end{equation}
which by \eqref{semidirectproduct} is readily seen to define a group action. Up to the isomorphism $\varphi$, this is also the action performed by $\mathbf{G}_1$ on $G$, via the definition
$$(g,h)\cdot_1 x := gxh^{-1};$$
namely $\varphi(\alpha, a)\cdot_1 x = (\alpha, a)\cdot_2 x$. Note the subscripts $1$ and $2$.

The semidirect product description $\mathbf{G}_2$ is convenient to express the stabilizer of the identity $e\in G$ under $\cdot_2$: it is simply $\{1\}\rtimes G \subset \mathbf{G}_2$. Its image under $\varphi$, which is the stabilizer of $e$ under the $\cdot_1$ action, coincides with the diagonal copy of $G$ in $G\times G=\mathbf{G}_1$. On the other hand, the direct product description $\mathbf{G}_1$ is convenient to express representations of an arbitrary group into $\mathbf{G}_1$: they are just pairs of representations into $G$. Both descriptions will be relevant in the sequel.

Consider finally $\mathbf{G}':=\mathfrak{g}\rtimes G$, where $G$ acts on $\mathfrak{g}$ by the adjoint representation. We can view $\mathbf{G}'$ as a version of $\mathbf{G}_2=G\rtimes G$ in which the left factor $G$ has been scaled up to infinity, becoming $\mathfrak{g}$. This is because the conjugation action of $G$ on~$G$ induces the adjoint action of $G$ on $T_{\mathrm{Id}}G\simeq \mathfrak{g}$. The group $\mathbf{G}'$ acts on $\mathfrak{g}$ by affine transformations, via
\begin{equation} (A,a)\cdot' x:=A+\mathrm{Ad}(a)(x) \label{affineaction} \end{equation}
which is the scaled-up version of the $\cdot_2$ action \eqref{semidirectaction}.

Two parallel and natural questions in the study of Lorentzian manifolds are: which subgroups of $\mathbf{G}_1\simeq \mathbf{G}_2$ act properly discontinuously on $G$? Which subgroups of $\mathbf{G}'$ act properly discontinuously on $\mathfrak{g}$? The quotients of $G$ are called complete Anti-de Sitter manifolds (Anti-de Sitter space is another name for $G$ endowed with its Killing form and isometry group $\mathbf{G}_2$); the quotients of $\mathfrak{g}$ are called Margulis spacetimes, as Margulis exhibited the first examples with action by a free group \cite{mar83, mar84}. 
The purpose of this note is to summarize and provide a reading guide for some recent results (\cite{GK, DGK1, DGK2}) on these quotients, showing in particular that quotients of~$\mathfrak{g}$ can be seen as ``geometric limits'', in a natural sense, of quotients of $G$. Rich context for Lorentzian manifolds can be found in the seminal paper \cite{mes90}.

\subsection{Plan of the paper}
Section \ref{basics} reviews the Lorentzian geometry of $G$ and $\mathfrak{g}$. Section \ref{contracting} gives a sufficient criterion for proper discontinuity of actions on $G$, in terms of \emph{contracting maps} for the metric on $\mathbb{H}^2$. An analogue for $\mathfrak{g}$ is also given. The main results, with sketches of proofs, can then be formulated in Section \ref{results}. In \ref{existence} we show the above criteria are also necessary. In \ref{limits} we study the transition from $G$ to $\mathfrak{g}$ in terms of the geometry of their quotients. In \ref{classification} we describe a natural classification of the quotients of $\mathfrak{g}$ in terms of a combinatorial object, the \emph{arc complex}.

\subsection*{Acknowledgments}
The material summarized in this note is mostly joint work with J.\ Danciger and F.\ Kassel. It is based on a series of lectures given at Almora (Uttarakhand) in December 2012, at the conference organized in the honor of R.\ Kulkarni's 70th birthday. It also owes much to discussions with T.\ Barbot, W.M.\ Goldman, V.\ Charette and T.\ Drumm. I am very indebted to Y.\ Minsky and F.\ \nolinebreak Labourie for sparking my initial interest in the subject. The Institut Henri Poincar\'e (Paris) and the Institut CNRS--Pauli (UMI 2842, Vienna) provided excellent working conditions.

\section{The pseudo-Riemannian geometry of $G$} \label{basics}

We begin with some remarks on the embedding $G\subset \mathbb{P}^3\mathbb{R}$ as the interior of a quadric $\partial_\infty G$. Figure \ref{AdSMink} shows the group $G$ in the same normalization as Figure \ref{GL2}, but with matrices deleted and some remarkable subsets highlighted instead (which we discuss below).

\begin{figure}[h!] \centering
%\psfrag{a}{$a$}
\psfrag{G}{$G$}
\psfrag{g}{$\mathfrak{g}$}
\psfrag{k}{$\mathfrak{k}$}
\psfrag{a}{$\mathfrak{a}$}
\psfrag{t}{$\mathfrak{t}$}
\psfrag{A}{$A$}
\psfrag{C}{$C$}
\psfrag{S}{$S$}
\psfrag{T}{$T$}
\psfrag{K}{$K$}
\psfrag{J}{$J$}
\psfrag{zoom in}{\scriptsize{zoom in}}
\psfrag{to Id}{\scriptsize{to $\mathrm{Id}$}}
\psfrag{I}{$:=\mathrm{Id}$}
\psfrag{d}{$\partial_\infty G$}
\includegraphics[width=11cm]{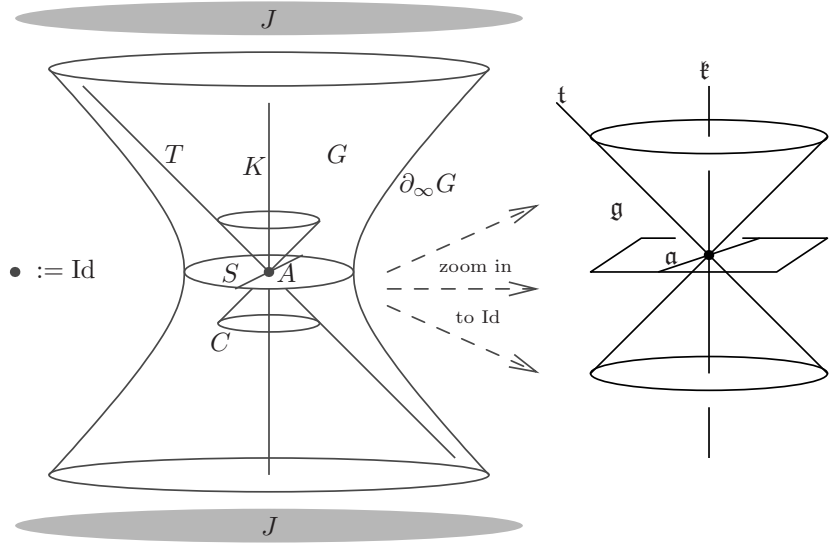}
\caption{Another view of $G=\mathrm{PSL}(2,\mathbb{R})$ and its tangent space $\mathfrak{g}$ at the identity element, marked by a dark point.} \label{AdSMink} 
\end{figure}

\smallskip

\subsection{Projective action}~

\noindent $\bullet$ The group $\mathbf{G}:=\mathbf{G}_1\simeq \mathbf{G}_2$ acts on $G$ by projective transformations. This is the neutral connected component of the full group of projective transformations of $\mathbb{P}^3\mathbb{R}$ preserving $G$ (a copy of the $6$-dimensional Lie subgroup $\mathrm{PO}(2,2;\mathbb{R})$ of the 15-dimensional projective group $\mathrm{PGL}(4,\mathbb{R})$). Here, the signature $(2,2)$ is that of the determinant $ad-bc$ seen as a quadratic form.

\smallskip 

\noindent $\bullet$ To visualize the action of the full group $\mathbf{G}$ on $G$ requires some effort, but we can describe the action of the stabilizer $1\rtimes G\subset \mathbf{G}_2$ of the identity. Namely, this action identifies with the action of $G$ on itself by conjugation, which globally preserves the projective plane at infinity (traceless matrices), and therefore restricts to an affine action on the complement of that projective plane. Since this action also fixes the identity, we can view it simply as a linear action on Figures \ref{GL2} and \ref{AdSMink}, preserving a one-sheeted hyperboloid $\mathbb{R}^3 \cap \partial_\infty G$ in $\mathbb{R}^3$. This linear action of $\{1\}\rtimes G$ passes to the tangent space $\mathfrak{g}$ at the identity, under the name of adjoint action (Figure \ref{AdSMink}).

\smallskip

\subsection{Remarkable subsets}~

\noindent $\bullet$ The space $\partial_\infty G$ of rank-one projective matrices is topologically a torus which comes with two natural foliations into projective lines: one line of the first foliation consists of all matrices with the same kernel; one line of the second foliation consists of all matrices with the same image. Some of these lines are drawn in Figure \ref{GL2}.

\smallskip 

\noindent $\bullet$ Symmetric matrices with positive determinant appear as the intersection $S$ of $G$ with a projective plane, namely the horizontal plane through the central element (identity matrix) in Figures \ref{GL2} and \ref{AdSMink}. This intersection is projectively equivalent to a round disk, and can be naturally identified with $\mathbb{H}^2$, as follows: a projective matrix belongs to $S$ if and only if it acts on $\mathbb{H}^2$ as a translation along an axis through the basepoint ($\sqrt{-1}$ in the upper half plane model). For each point $p\in \mathbb{H}^2$ there is a unique such translation taking $\sqrt{-1}$ to $p$. This provides the identification $S\simeq \mathbb{H}^2$.

\smallskip 

\noindent $\bullet$ The group $K$ of rotation matrices (fixing $\sqrt{-1}$ in $\mathbb{H}^2$) is the projective line of matrices whose diagonal entries are equal and whose nondiagonal entries are opposite. This line $K$ appears vertical in Figure \ref{AdSMink}; it closes up at infinity and is disjoint from $\partial_\infty G$. The Lie algebra $\mathfrak{k}$ of $K$ is a line in $\mathfrak{g}$.

\smallskip 

\noindent $\bullet$ The group $A$ of diagonal matrices (translations along the line $(0,\infty)$ through $\sqrt{-1}$ in $\mathbb{H}^2$) is the projective open segment of matrices whose nondiagonal entries are $0$, and whose diagonal entries have the same sign. This segment $A$ appears in Figure \ref{AdSMink} as a horizontal segment through the identity matrix, contained in $S$, directed towards the viewer. The Lie algebra $\mathfrak{a}$ of $A$ is a line in $\mathfrak{g}$.

\smallskip 

\noindent $\bullet$ The parabolic elements of $G$ form a cone $C$ (or degenerate quadric) with vertex at the identity element. This cone is shown in Figures \ref{GL2} and \ref{AdSMink}, containing for example the group $T$ of lower triangular matrices as a projective line. The Lie algebra $\mathfrak{t}$ of $T$ is again a line in $\mathfrak{g}$.

\smallskip 

\noindent $\bullet$ The space $J$ of projectivized, traceless matrices with positive determinant is the intersection of $G$ with a projective plane (the plane at infinity in Figures \ref{GL2} and \ref{AdSMink}), and is again projectively equivalent to a round disk. A matrix in $G$ is traceless if and only if it acts on $\mathbb{H}^2$ as a rotation of angle $\pi$; for any $p\in \mathbb{H}^2$ there is a unique such rotation taking the basepoint $\sqrt{-1}$ to $p$: thus, traceless elements of $G$ identify again with $\mathbb{H}^2$. In fact, multiplication by $(^0_1{~}^{\text{-1}}_{\,0})\in K$ (on either side) swaps the copies $J$ and $S$ of $\mathbb{H}^2$. 

\smallskip 

\subsection{Lorentzian metric on $G$}~

\noindent $\bullet$ Through two distinct points $a,b$ of $G$ passes a unique projective line. If this line intersects $\partial_\infty G$ in two distinct points $a', b'$, with $a', a, b, b'$ in cyclic order, then the cross-ratio $[a':a:b:b']$ (defined so that $[0:1:t:\infty]=t$) is a number in $(1,+\infty)$. In analogy with the definition of the hyperbolic distance between points of the projective model of hyperbolic space, we may define $$\delta(a,b):=\frac{1}{2} \log [a':a:b:b']>0~.$$ This function $\delta$ is \emph{not} a metric on $G$, but it is preserved by the action of $\mathbf{G}$: it is helpful to speak of $\delta(a,b)$ as the ``Lorentzian distance'' from $a$ to $b$. When restricted to a round section of $G$ such as the space of symmetric matrices $S$ or of traceless matrices $J$, the function $\delta$ does define a metric, isometric to $\mathbb{H}^2$.

\smallskip 

\noindent $\bullet$ If the projective line through $a,b\in G$ is disjoint from $\partial_\infty G$, then applying the same definition to the two \emph{imaginary} intersection points of that line with $\partial_\infty G$ leads to a number $\delta(a,b)$ that is pure imaginary, and is defined only modulo $i\pi$ and sign. If the projective line through $a,b\in G$ is tangent to $\partial_\infty G$, we can define $\delta(a,b):=0$ to make $\delta$ continuous on $G\times G$. 

\smallskip

\noindent $\bullet$ The points of $G$ whose Lorentzian distance $\delta$ to the identity matrix is pure imaginary, say $i\theta$ with $\pm \theta\in[0,\frac{\pi}{2}]$, are the matrices acting on $\mathbb{H}^2$ as rotations of angle $2\theta$. Equivalently, they are the matrices with trace $\pm 2\cos \theta \in (-2,2)$. Equivalently still, they are the matrices inside the lightlike cone $C$ pictured in Figure \ref{AdSMink}. Matrices on $C$ are parabolic isometries of $\mathbb{H}^2$ (of trace $\pm 2$, fixing exactly one point at infinity). Matrices outside $C$ are hyperbolic translations of $\mathbb{H}^2$, fixing two points at infinity. They are the matrices whose Lorentzian distance $\delta$ to the identity matrix is real positive, say equal to some $\lambda>0$: their trace is then $2\cosh \lambda$ as they translate $\mathbb{H}^2$ by a length $2\lambda$. More generally, $$\delta(a,b)=\mathrm{arccosh}\left (\frac{1}{2}|\mathrm{Tr}(a^{-1}b)|\right )$$ for all $a,b\in G$, with the natural convention $\mathrm{arccosh}(t)=i\, \mathrm{arccos}(t)$ when $0\leq t < 1$. Note that the total length of the loop $K$ is then $i\pi$ (not $2i\pi$).

\smallskip 

\noindent $\bullet$ The space of all projective lines intersecting $G$ has exactly three orbits, characterized by the number of intersection points with $\partial_\infty G$: either $2$ (orbit of $A$ or ``spacelike'' lines), or $0$ (orbit of $K$ or ``timelike'' lines), or $1$ (orbit of $T$, which is tangent to $\partial_\infty G$ at infinity, or ``lightlike'' lines). In terms of the Lorentzian distance $\delta$ on $G$, there is nothing special about the identity matrix: any point $g\in G$ is the vertex of a light cone, defined as the union of all lines through $g$ tangent to $\partial_\infty G$.

\smallskip

\noindent $\bullet$ Intuitively, if one rescales the Lorentzian distance $\delta$ on $G$ while zooming in near the identity, the resulting object is the square root of (a multiple of) the standard flat Lorentzian form on $\mathfrak{g}\simeq\mathbb{R}^3$, also known as the Killing form.

\smallskip

\section{Contracting maps and properly discontinuous actions} \label{contracting}
This section is a review of \cite[\S 6.2]{DGK1}: we explain the link between contracting maps and properly discontinuous actions, in the ``easy'' direction.
\subsection{Timelike lines and maps $\mathbb{H}^2\rightarrow \mathbb{H}^2$}~
Since the group $K$ of rotation matrices is the stabilizer of the basepoint $\sqrt{-1}$ of $\mathbb{H}^2$, any double coset $gKh^{-1}=(g,h)\cdot_1 K$ of $K$ (where $(g, h)\in G\times G = \mathbf{G}_1$) can be seen as the space of all orientation-preserving isometries of $\mathbb{H}^2$ taking $q:=h(\sqrt{-1})$ to $p:=g(\sqrt{-1})$. Such a double coset will therefore be written $${}_pK_q:=\{g\in G~|~ g(q)=p\}~,$$ so that for example ${}_pK_q\; {}_qK_r = {}_pK_r$.  Here is an important result:

\begin{proposition} \label{Lip}
Suppose $f:\mathbb{H}^2 \rightarrow \mathbb{H}^2$ is a $C$-Lipschitz map with $C<1$. Then the collection of lines $\{_{f(p)}K_p\}_{p\in \mathbb{H}^2}$ form a fibration of $G$.
\end{proposition}
\begin{proof}
Note that ${}_{p'}K_p$ intersects ${}_{q'}K_q$ if and only if there exists an isometry of~$\mathbb{H}^2$ taking $p$ to $p'$ and $q$ to $q'$, or equivalently, if $d_{\mathbb{H}^2}(p,q)=d_{\mathbb{H}^2}(p',q')$. This never happens for $p'=f(p)$ and $q'=f(q)$ (unless $p=q$), because $f$ contracts all distances. Therefore, all lines $_{f(p)}K_p$ are pairwise disjoint. In fact, an element $g$ of $G$ belongs to $_{f(p)}K_p$ if and only if $g(p)=f(p)$, or equivalently, if $g^{-1}\circ f$ fixes~$p$. But $g^{-1}\circ f$ is $C$-Lipschitz because $g$ is an isometry of $\mathbb{H}^2$: therefore $g^{-1}\circ f$ fixes precisely one point of $\mathbb{H}^2$, which means that $g$ lies on precisely one line  $_{f(p)}K_p$. The lines $\{_{f(p)}K_p\}_{p\in \mathbb{H}^2}$ thus form a partition of $G$. 

Finally, the line that contains $g$ depends continuously on $g$: this amounts to the fact that the fixed point of $g^{-1}\circ f$ varies continuously with $g$. To check this fact, fix $\varepsilon>0$. If $h\in G$ is close enough to $g$ in the sense that it takes the fixed point $p$ of $g^{-1}\circ f$ within distance $(1-C)\varepsilon$ from $g(p)=f(p)$, then $d_{\mathbb{H}^2}(p,h^{-1}\circ f (p))\leq (1-C)\varepsilon$, which implies that the $\varepsilon$-ball centered at $p$ is stable under $h^{-1}\circ f$. Therefore this ball contains the fixed point of $h^{-1}\circ f$: this finishes the proof.
\end{proof}

We now derive an infinitesimal analogue.

The Lie algebra $\mathfrak{g}$ of $G$ can be seen as the space of all Killing fields on $\mathbb{H}^2$, where a Killing field $X$ is by definition a vector field whose flow is a one-parameter group of isometries (these isometries of $\mathbb{H}^2$ may be elliptic, parabolic or hyperbolic, depending on the position of $X$ relative to the infinitesimal light cone in $\mathfrak{g}$, right panel of Figure \ref{AdSMink}). The Lie subalgebra $\mathfrak{k}$ consists of all Killing fields that vanish at the basepoint $\sqrt{-1}$ of $\mathbb{H}^2$. Therefore, for any $(A,a)\in \mathbf{G}'=\mathfrak{g}\rtimes G$, the space $(A,a)\cdot' \mathfrak{k}=A+\mathrm{Ad}(a)(\mathfrak{k})\subset \mathfrak{g}$ consists of all Killing fields that coincide with $A$ at the point $p:=a(\sqrt{-1})$. If $v:=A(p)\in T_p\mathbb{H}^2$, then such a $\mathbf{G}'$-translate of $\mathfrak{k}$ can be written $${}_v \mathfrak{k}_p:=\{X\in\mathfrak{g}~|~X(p)=v\}$$ in analogy with the macroscopic case: it is the space of all Killing fields taking the value $v$ at the point $p$.

For $c\in \mathbb{R}$, say that a vector field $Y$ on $\mathbb{H}^2$ is \emph{$c$-lipschitz} (lowercase! to distinguish from the notion of a Lipschitz map between metric spaces) if 
\begin{equation} \label{dprime} \left . \frac{\mathrm{d}}{\mathrm{d}t} \right |_{t=0} d\left (\exp_p(tY(p)), \exp_q(tY(q))\right )\leq c d(p,q) \end{equation}
for all distinct $p,q\in \mathbb{H}^2$, where $\exp_p:T_p\mathbb{H}^2 \rightarrow \mathbb{H}^2$ is the exponential map. For example, a Killing field is $0$-lipschitz. Intuitively, a vector field is $c$-lipschitz when it is the right derivative at $t=0$ of a one-parameter family of deformations of the identity $f_t:\mathbb{H}^2\rightarrow \mathbb{H}^2$, with $f_0=\mathrm{Id}_{\mathbb{H}^2}$ and $f_t$ Lipschitz of constant $1+ct$. We allow $c<0$.

In $\mathbb{H}^2$, a continuous, $c$-lipschitz vector field $Y$ with $c<0$ always has a unique zero: indeed, using \eqref{dprime} we see that $Y$ is inward-pointing on any ball centered at a point $p\in\mathbb{H}^2$ of radius larger than $\Vert Y(p)\Vert /|c|$, hence $Y$ admits a zero in this ball by Brouwer's theorem. Uniqueness of the zero follows again from \eqref{dprime} and the assumption $c<0$.

\begin{proposition} \label{lip}
Suppose $Y$ is a $c$-lipschitz, continuous vector field on $\mathbb{H}^2$ with $c<0$. Then the lines $\{_{Y( p)}\mathfrak{k}_p\}_{p\in \mathbb{H}^2}$ form a fibration of $\mathfrak{g}$.
\end{proposition}
\begin{proof}
We mimic the proof of Proposition \ref{Lip}.
For $(p,v), (q,w)\in T\mathbb{H}^2$ with $p\neq q$, note that ${}_{v}\mathfrak{k}_p$ intersects ${}_{w}\mathfrak{k}_q$ if and only if there exists a Killing field on $\mathbb{H}^2$ that agrees with $v$ at $p$ and with $w$ at $q$, or equivalently, if $d_{\mathbb{H}^2}(\exp_p(tv),\exp_q(tw))=o(t)$. This never happens for $v=Y(p)$ and $w=Y(q)$, because $Y$ contracts all distances by~\eqref{dprime}. Therefore, all lines $_{Y(p)}\mathfrak{k}_p$ are pairwise disjoint. In fact, a Killing field $X\in \mathfrak{g}$ belongs to $_{Y(p)}\mathfrak{k}_p$ if and only if $X(p)=Y(p)$, or equivalently, if $Y-X$ vanishes at~$p$. But $Y-X$ is $c$-lipschitz because $X$ is a Killing field of $\mathbb{H}^2$: therefore $Y-X$ vanishes at precisely one point of $\mathbb{H}^2$, which means that $X$ lies on precisely one line  $_{Y(p)}\mathfrak{k}_p$. The lines $\{_{Y(p)}\mathfrak{k}_p\}_{p\in \mathbb{H}^2}$ thus form a partition of $\mathfrak{g}$. 

Finally, the line that contains $X\in \mathfrak{g}$ depends continuously on $X$: by continuity of $Y$, this amounts to the fact that the zero of $Y-X$ varies continuously with $X$. To check this fact, fix $\varepsilon>0$. If $X'\in \mathfrak{g}$ is close enough to $X$ in the sense that $\Vert X'(p)-X(p)\Vert < |c|\varepsilon$ (where $p\in \mathbb{H}^2$ is the unique zero of the vector field $Y-X$), then \eqref{dprime} shows that $Y-X'$ is inward-pointing on the $\varepsilon$-ball centered at $p$. Therefore this ball contains the unique zero of $Y-X'$: this finishes the proof.
\end{proof}

\subsection{Equivariance and properly discontinuous actions}
Propositions~\ref{Lip} and~\ref{lip} have an interesting upshot when the map $f:\mathbb{H}^2 \rightarrow \mathbb{H}^2$ or the continuous vector field $Y$ on $\mathbb{H}^2$ have an equivariance property.

Suppose $\Gamma$ is a free group or surface group and $j:\Gamma\rightarrow G=\mathrm{Isom}_0(\mathbb{H}^2)$ a Fuchsian representation. In other words, $j(\Gamma)\backslash \mathbb{H}^2$ is a hyperbolic surface, possibly with cusps and/or funnels. Suppose $\rho:\Gamma\rightarrow G$ is another representation, not necessarily Fuchsian. We may combine $\rho, j$ into a single representation 
$$(\rho, j):\Gamma\rightarrow \mathbf{G}_1=G\times G~.$$
We say that a map $f:\mathbb{H}^2 \rightarrow \mathbb{H}^2$ is $(j,\rho)$-equivariant if for all $p\in \mathbb{H}^2$ and $\gamma\in \Gamma$,
$$f(j(\gamma)\cdot p)=\rho(\gamma)\cdot f(p)~.$$
\begin{proposition} \label{Equiv}
Suppose there exists a $(j,\rho)$-equivariant, $C$-Lipschitz map $f:\mathbb{H}^2\rightarrow \mathbb{H}^2$ with $C<1$. Then the line fibration $\{{}_{f(p)}K_p\}_{p\in\mathbb{H}^2}$ of $G$ given by Proposition \ref{Lip} is invariant under the $\cdot_1$-action of $(\rho,j)(\Gamma)\subset \mathbf{G}_1$ on $G$, which must be properly discontinuous.
\end{proposition}
\begin{proof}
An isometry $g\in G$ of $\mathbb{H}^2$ takes $p$ to $f(p)$ if and only if $(\rho, j)(\gamma)\cdot_1 g = \rho(\gamma)gj(\gamma)^{-1}$ takes $j(\gamma)\cdot p$ to $\rho(\gamma)\cdot f(p)$. But the latter point is $f(j(\gamma)\cdot p)$: this shows 
$$(\rho, j)(\gamma)\cdot_1 \, {}_{f(p)}K_p={}_{f(j(\gamma)\cdot p)}K_{j(\gamma)\cdot p}~,$$ for all $\gamma\in \Gamma$ and $p\in \mathbb{H}^2$. We have in fact just proved that the fibration $\pi:G\rightarrow \mathbb{H}^2$ taking $g\in G$ to the fixed point of $g^{-1}\circ f$ in $\mathbb{H}^2$ is $((\rho, j), j)$-equivariant. This proves proper discontinuity of the $(\rho,j)(\Gamma)\cdot_1$-action on the total space $G$, because $j(\Gamma)$ acts properly discontinuously on the base $\mathbb{H}^2$.
\end{proof}

We now develop an infinitesimal analogue (see \cite{gm00} for a related result). Suppose that $\Gamma$ is a free group and $j:\Gamma\rightarrow G=\mathrm{Isom}_0(\mathbb{H}^2)$ a Fuchsian representation. Suppose $u:\Gamma\rightarrow \mathfrak{g}$ is a $j$-\emph{cocycle}, that is, a map such that
$$(u,j)(\gamma):=(u(\gamma), j(\gamma))\in \mathfrak{g}\rtimes G = \mathbf{G}'$$
defines a representation $(u,j):\Gamma\rightarrow \mathbf{G}'$. Intuitively, $u$ measures the difference between $j$ and another, infinitesimally close representation $\rho$.

We say that a vector field $Y$ on $\mathbb{H}^2$ is $(j,u)$-equivariant if for all $p\in \mathbb{H}^2$ and $\gamma\in \Gamma$,
$$Y(j(\gamma)\cdot p)=j(\gamma)_*(Y(p))+u(\gamma)(j(\gamma)\cdot p)~.$$Intuitively, this means that $Y$ looks like the derivative at time $t=0$ of a smooth family of $(j,\rho_t)$-equivariant maps $f_t:\mathbb{H}^2 \rightarrow \mathbb{H}^2$, with $\rho_0=j$ and $f_0=\mathrm{Id}_{\mathbb{H}^2}$, such that $(\rho_t)_{t\geq 0}$ departs from $j$ in the direction $u$. Note that the adjoint action of $G$ on $\mathfrak{g}$ is the same as the pushforward action of $\mathrm{Isom}_0(\mathbb{H}^2)$ on Killing fields: for all $g\in G$ and $X\in \mathfrak{g}$,
$$([\mathrm{Ad}(g)](X))(g\cdot p)=g_*(X(p)) \in T_{g\cdot p}\mathbb{H}^2~.$$
\begin{proposition} \label{equiv}
Suppose there exists a $(j,u)$-equivariant, continuous, $c$-lipschitz vector field $Y$ on $\mathbb{H}^2$ with $c<0$. Then the line fibration $\{{}_{Y(p)}\mathfrak{k}_p\}_{p\in\mathbb{H}^2}$ of $\mathfrak{g}$ given by Proposition~\ref{lip} is invariant under the $\cdot'$-action of $(u,j)(\Gamma)\subset \mathbf{G}'$ on $\mathfrak{g}$, which must be properly discontinuous.
\end{proposition}
\begin{proof}
We mimic the proof of Proposition \ref{Equiv}. A Killing field $X\in \mathfrak{g}$ on $\mathbb{H}^2$ takes $p$ to $Y(p)\in T_p\mathbb{H}^2$ if and only if $(u,j)(\gamma)\cdot' X = u(\gamma) + \mathrm{Ad}(j(\gamma))(X)$ takes $j(\gamma)\cdot p$ to $u(\gamma) (j(\gamma)\cdot p) + j(\gamma)_*(Y(p))$. But the latter point is $Y(j(\gamma)\cdot p)$: this shows 
$$(u, j)(\gamma)\cdot' \, {}_{Y(p)}\mathfrak{k}_p={}_{Y(j(\gamma)\cdot p)}\mathfrak{k}_{j(\gamma)\cdot p}~,$$ for all $\gamma\in \Gamma$ and $p\in \mathbb{H}^2$. We have in fact just proved that the fibration $\pi':\mathfrak{g}\rightarrow \mathbb{H}^2$ taking $X\in \mathfrak{g}$ to the zero of $Y-X$ is $((u, j), j)$-equivariant. This proves proper discontinuity of the $(u,j)(\Gamma)\cdot'$-action on the total space $\mathfrak{g}$, because $j(\Gamma)$ acts properly discontinuously on the base $\mathbb{H}^2$.
\end{proof}
The reason we take for $\Gamma$ a free group (not a surface group) in Proposition \ref{equiv} is that for surface groups, no cocycle $u$ can admit an equivariant, $c$-lipschitz vector field with $c<0$: indeed the flow of such a vector field decreases areas to first order, whereas one can show equivariant vector fields must preserve total area to first order (because the area of a closed hyperbolic surface depends only on its genus). See also \cite[\S 2]{stretchmap}.

Note that under the conditions of Propositions \ref{Equiv} and \ref{equiv}, the geodesic fibrations descend to the quotient Lorentzian manifolds $(\rho, j)(\Gamma)\backslash G$ and $(u, j)(\Gamma)\backslash \mathfrak{g}$.

Are there other quotients of $G$ or $\mathfrak{g}$? Essentially no, up to finite index.
It actually follows from work of Kulkarni--Raymond \cite{kr85} that \emph{any} manifold that is a quotient of $G$ by a finitely generated group is virtually of the form given in Proposition~\ref{Equiv}, with $\Gamma$ a free group or surface group and $j:\Gamma\rightarrow G$ Fuchsian. Similarly, any manifold that is a quotient of $\mathfrak{g}$ by a finitely generated group is virtually of the form given in Proposition \ref{equiv}, with $\Gamma$ free and $j$ Fuchsian, unless the group is virtually solvable: this follows from work of Fried--Goldman \cite{fg83}.

\section{Main results} \label{results}
\subsection{Existence of contracting maps and fields} \label{existence}
It turns out that Propositions \ref{Equiv} and \ref{equiv} admit a converse ($\Gamma$ is still a finitely generated free group or surface group):

\begin{theorem} \label{mainmacro}
Suppose $(\rho, j):\Gamma\rightarrow \mathbf{G}_1=G\times G$ is a representation, with $j:\Gamma\rightarrow G$ Fuchsian. We assume that $j(\beta)$ is a hyperbolic translation of greater length than $\rho(\beta)$, for at least some $\beta\in \Gamma$ (where the length of an isometry of $\mathbb{H}^2$ other than a translation is defined to be $0$). The $\cdot_1$-action of $(\rho, j)(\Gamma)$ on $G$ is properly discontinuous if and only if there exists a $(j,\rho)$-equivariant, $C$-Lipschitz map $f:\mathbb{H}^2\rightarrow \mathbb{H}^2$ with $C<1$. 
\end{theorem}
Note that the assumption on lengths is always satisfied, up to exchanging $j$ and~$\rho$ (if no $\rho(\gamma)$ has shorter translation length than $j(\gamma)$, then $\rho$ is Fuchsian as well as~$j$). This theorem is proved in \cite{kasPhD} for $j$ convex-cocompact, and in \cite{GK} in full generality (and also higher dimension).
Here is the infinitesimal analogue.
\begin{theorem} \cite[Th.\ 1.1]{DGK2} \label{mainmicro}
Suppose $(u, j):\Gamma\rightarrow \mathbf{G}'=\mathfrak{g} \rtimes G$ is a representation, with $j:\Gamma\rightarrow G$ Fuchsian and convex-cocompact. The $\cdot'$-action of $(u, j)(\Gamma)$ on $\mathfrak{g}$ is properly discontinuous if and only if, up to exchanging $u$ with $-u$, there exists a $(j,u)$-equivariant, $c$-lipschitz, smooth vector field $Y$ on $\mathbb{H}^2$ with $c<0$. 
\end{theorem}
By Proposition \ref{equiv}, the quotients of $\mathfrak{g}$ arising from Theorem \ref{mainmicro} are (trivial) line bundles over surfaces with boundary, and therefore, are homeomorphic to open handlebodies. In particular, they are topologically tame, a fact obtained independently by Choi and Goldman \cite{cg13} via very different methods.

The convex-cocompact assumption in Theorem \ref{mainmicro} should be unnecessary, but makes proofs more straightforward.
We will only sketch a proof of these two theorems. For Theorem \ref{mainmacro}, the strategy can be broken up in several steps (we refer to \cite[\S 5]{GK} for details). 

\emph{Step 1:} Consider the infimum $C_0$ of possible Lipschitz constants for $(j,\rho)$-equivariant maps. This infimum is achieved by some map $f_0$, by the Ascoli-Arzel\`a theorem, except possibly in some degenerate cases (when $\rho$ fixes exactly one point at infinity of $\mathbb{H}^2$), which can be treated separately. By Proposition \ref{Equiv}, we only need to prove that if $C_0\geq 1$ then $(j,\rho)(\Gamma)$ does not act properly discontinuously on $G$.

\emph{Step 2:} Let $\mathcal{F}$ denote the space of all $C_0$-Lipschitz, $(j,\rho)$-equivariant maps $f:\mathbb{H}^2\rightarrow \mathbb{H}^2$. For all $f\in\mathcal{F}$, let $E_f\subset \mathbb{H}^2$ be the (closed) set of all points $p$ such that for every neighborhood $U$ of $p$ in $\mathbb{H}^2$, the smallest Lipschitz constant of the restriction $f|_U$ is still $C_0$. We call $E_f$ the \emph{stretch locus} of $f$; one should think of $f$ as ``good'' when $E_f$ is small. A compactness argument shows that $E_f$ is not empty.

The space $\mathcal{F}$ is convex for pointwise geodesic interpolation: if $f_0, f_1\in \mathcal{F}$ then for all $t\in [0,1]$, the map $f_t$ taking each $p\in \mathbb{H}^2$ to the unique point at distance $t\, d(f_0(p), f_1(p))$ from $f_0(p)$ and at distance $(1-t)\, d(f_0(p), f_1(p))$ from $f_1(p)$, is in~$\mathcal{F}$. This follows from the convexity, with respect to time, of the distance between two points travelling at constant velocities on geodesics of $\mathbb{H}^2$, itself a consequence of $\mathrm{CAT}(0)$--thinness of hyperbolic triangles. By the same argument, 
$$E_{f_t} \subset E_{f_0} \cap E_{f_1}$$
for all $0<t<1$. Iterating this barycentric construction for a sequence $(f_i)_{i\in\mathbb{N}}$ that is, in an appropriate sense, dense in $\mathcal{F}$, produces a map $f$ with smallest possible stretch locus: 
$$E_f = \bigcap_{\varphi\in \mathcal{F}} E_\varphi~.$$ 

\emph{Step 3:} Suppose first that $C_0>1$. 
We can then prove that $E_f$ is a geodesic lamination, \emph{i.e.} contains a unique germ of line through any of its points, and that $f$ takes any segment of any of those lines to a segment $C_0$ times longer (we say that the segment is \emph{$C_0$-stretched}). This is closely related to results of Thurston \cite{stretchmap}, though our approach is quite different. The idea is to use the following theorem of Kirszbraun and Valentine \cite{Kirszbraun34, Valentine44}:
\begin{theorem} \label{kirszbraun}
Any $C_0$-Lipschitz map from a subset of $\mathbb{H}^2$ to $\mathbb{H}^2$, with $C_0\geq 1$, has a $C_0$-Lipschitz extension to $\mathbb{H}^2$.
\end{theorem}
To prove that $E_f$ is a lamination, we apply this theorem in the following way. Pick a small ball $B$ of radius~$r$ centered at $p\in E_f$. 
Any $C_0$-Lipschitz extension $f'$ of the restriction $f|_{\partial B}$ to $\{p\}\cup \partial B$ satisfies $$\max_{q\in \partial B}d(f'(p),f'(q))/r=C_0~:$$ otherwise, a constant extension of $f'$ near $p$ would still be $C_0$-Lipschitz, and pasting this map with $f|_{\mathbb{H}^2\smallsetminus B}$ we could apply the theorem to find a map $f''\in \mathcal{F}$ with $p\notin E_{f''}$: absurd. Therefore, $$Q:=\{q\in \partial B~|~d(f'(p),f'(q))=rC_0\}$$ is not empty: at least some germs of \emph{rays} issued from $p$ are $C_0$-stretched. No half-plane of $T_{f'(p)}\mathbb{H}^2$ contains the directions of the images of all these rays, for otherwise we could push $f'(p)$ into such a half-plane and again decrease the stretch locus. 
To find a germ of \emph{line}, we must prove that $Q$ consists of only a pair of opposite points. This can be seen by refining the argument above, and using $$d(\exp_p(C_0 rv), \exp_p(C_0 rv'))> C_0\, d(\exp_p(rv), \exp_p(rv'))$$ for all independent unit vectors $v,v'\in T_p\mathbb{H}^2$ (again a consequence of convexity of the distance function, this is sometimes called Toponogov's inequality).

\emph{Step 4:} (see also \cite[\S 7]{GK}) The geodesic lamination $E_f\subset \mathbb{H}^2$ projects to a geodesic lamination $\Lambda$ of the quotient surface $j(\Gamma)\backslash \mathbb{H}^2$. Suppose for simplicity that $\Lambda$ contains a closed leaf, corresponding to the conjugacy class of an element $\gamma\in \Gamma$: then $\rho(\gamma)$ has ($C_0$ times) greater translation length than $j(\gamma)$, as a translation of $\mathbb{H}^2$. But by assumption, there exists $\beta\in \Gamma$ such that $\rho(\beta)$ has smaller translation length than $j(\beta)$. It is then not hard to find a sequence $\psi(n)\rightarrow \infty$ such that 
\begin{equation}
\mu(\rho(\beta^n \gamma^{\psi(n)}))=\mu(j(\beta^n \gamma^{\psi(n)}))+O(1)~, \label{mumu} 
\end{equation}
where by definition $\mu(g)=d(\sqrt{-1},g\cdot \sqrt{-1})$ for all $g\in G$ and $\sqrt{-1}$ is the basepoint of $\mathbb{H}^2$ (this follows from $\mu(g^{-1})=\mu(g)$, the triangle inequality $\mu(gh)\leq \mu(g)+\mu(h)$, and the fact that $\mu(g^n)/n$ limits to the translation length of $g$). The relationship~\eqref{mumu} shows that $(\rho,j)(\Gamma)$ cannot act properly discontinuously on $G$: indeed $\gamma_n:=\beta^n \gamma^{\psi(n)}$ has the property that the set $$\rho(\gamma_n)C_Rj(\gamma_n)^{-1}=(\rho, j)(\gamma_n)\cdot_1 C_R$$ intersects $C_R$ for all $n$, where by definition $C_R$ is the compact subset of $G$ consisting of all elements $g$ such that $\mu(g)\leq R$ (and $R$ is chosen larger than the error term in \eqref{mumu}).

\emph{Step 5:} It remains to remove the simplifying assumptions made above. If the geodesic lamination $\Lambda$ does \emph{not} contain any simple closed curves, then we can still follow a leaf of $\Lambda$ until it nearly closes up, thus defining the conjugacy class of an element $\gamma\in \Gamma$ such that $\rho(\gamma)$ has greater translation length than $j(\gamma)$ (as a translation of $\mathbb{H}^2$), and apply the same argument. Finally, if $C_0=1$, a more thorough discussion of the stretch locus $E_f$ shows that it consists of the union of a geodesic lamination $\Lambda$ and (possibly) some complementary components: elements $\gamma_n\in \Gamma$ defined by following leaves of $\Lambda$ for a great length satisfy directly the \eqref{mumu}-like relationship $\mu(\rho(\gamma_n))=\mu(j(\gamma_n))+O(1)$ that violates proper discontinuity of the $(\rho,j)(\Gamma)\cdot_1$-action on $G$.

\smallskip

For Theorem \ref{mainmicro}, the strategy is similar: define the infimum $c_0$ of possible lipschitz constants for $(u,j)$-equivariant vector fields on $\mathbb{H}^2$, find ``optimal'' vector fields realizing $c_0$, and define and study their (infinitesimal) stretch locus. However, we point out some key differences (referring to \cite{DGK1} for details):
\begin{itemize}
\item In Step 1, there are no ``degenerate cases'': the Ascoli-Arzel\`a theorem always provides a minimizing vector field as a pointwise limit.
\item However, the minimizer tends to be in general discontinuous or multi-valued (a closed subset of $T\mathbb{H}^2$ containing possibly more than one vector above some points of $\mathbb{H}^2$). This comes from the fact that in the relationship \eqref{dprime} defining lipschitz vector fields, nothing prevents vectors at nearby points from pointing strongly towards one another. A ``natural'' minimizer is in general a convex-set-valued section of $T\mathbb{H}^2$, or \emph{convex field} $Y\subset T\mathbb{H}^2$, with $Y$ closed 
(\cite[\S 3]{DGK1}).
\item For similar reasons, outside the lift $\Omega\subset \mathbb{H}^2$ of the convex core, adding to $Y$ an arbitrarily strong component towards $\Omega$ would never increase the lipschitz constant of $Y$: therefore we technically must enforce some standard behavior for $Y$ outside $\Omega$ to ensure well-definedness \cite[\S 4.2]{DGK1}.
\item One can then develop the whole infinitesimal theory, including an analogue of Theorem \ref{kirszbraun}, for convex fields, and find that the $(u,j)(\Gamma)\cdot'$-action on $\mathfrak{g}$ is properly discontinuous if and only if $Y$ has negative lipschitz constant $c$.
\item It remains to regularize $Y$. To turn $Y$ into a continuous \emph{vector} field \cite[\S 5.4]{DGK1} (in fact, a Lipschitz section of the tangent bundle of $\mathbb{H}^2$) with only minor damage to the constant $c$, we write $Y=Y_0+Y_1$ where $Y_0$ is a smooth, $(u,j)$-equivariant vector field. Then the convex field $Y_1\subset T\mathbb{H}^2$ is $j$-invariant, which means $j(\gamma)_* Y_1=Y_1$ for all $\gamma\in\Gamma$. We then replace $Y_1$ with a (still $j$-invariant) vector field $Y_1^t$ obtained by flowing each vector of $Y_1$ for a small negative time $t<0$ under the geodesic flow of $\mathbb{H}^2$. The one-dimensional case helps understand why $Y_1^t$ is a vector field when $Y_1$ is a convex field: by \eqref{dprime}, a $c$-lipschitz convex field on $\mathbb{R}\simeq\mathbb{H}^1$ identifies with a curve in $T\mathbb{R}\simeq \mathbb{R}\oplus\mathbb{R}$ such that the segment between any two of its points is either vertical, or has slope $\leq c$. The backwards flow amounts to applying a linear transformation $(^1_0 {}^t_1)$ to that curve, turning it into the graph of a Lipschitz \emph{function} on $\mathbb{R}$ with variation rates in $[\frac{1}{t},\frac{c}{1+ct}]$ (recall $t$ is a small negative number). 
\item Finally, a convolution procedure \cite[\S 5.5]{DGK1} allows us turn the continuous vector field $Y_1^t$ into a smooth (still $j$-invariant) vector field $Z$, still keeping the lipschitz constant of $Y_0+Z$ negative.
\end{itemize} 

\subsection{Geometric limits} \label{limits}

An important upshot of this last smoothness property is that it lets us realize the intuitive picture of contracting vector fields as derivatives of contracting maps:

\begin{theorem} \cite[Th.\ 1.5]{DGK1} \label{integrate}
Suppose $(u, j):\Gamma\rightarrow \mathbf{G}'=\mathfrak{g} \rtimes G$ satisfies the hypotheses of Theorem \ref{mainmicro}, and the $\cdot'$-action of $(u, j)(\Gamma)$ on $\mathfrak{g}$ is properly discontinuous. Suppose $(\rho_t)_{t\geq 0}$ is a smooth family of representations $\rho_t:\Gamma\rightarrow G$ with $\rho_0=j$, and that $\rho_t$ departs $j$ in the direction $u$ (namely, 
$\rho_t(\gamma)=e^{tu(\gamma)+o(t)}j(\gamma)$
%$\left . \frac{\mathrm{d}}{\mathrm{d}t} \right |_{t=0}\rho_t(\gamma)j(\gamma)^{-1}=u(\gamma)$ in $\mathfrak{g}$, 
for all $\gamma\in \Gamma$). Suppose further that $\rho_t(\gamma)$ has shorter hyperbolic translation length than $j(\gamma)$, for at least some $\gamma\in \Gamma$ and $t>0$. Then there exist $c<0$, a $(u,j)$-equivariant smooth vector field $Y$ on $\mathbb{H}^2$, and a family of $(j,\rho_t)$-equivariant maps $f_t:\mathbb{H}^2\rightarrow \mathbb{H}^2$ such that $f_t$ has Lipschitz constant at most $1+ct+o(t)$, and $\left . \frac{\mathrm{d}}{\mathrm{d}t}\right |_{t=0}f_t(p)=Y(p)$ for all $p\in \mathbb{H}^2$.
\end{theorem}

Theorem \ref{integrate} means that we can make the geodesic fibrations $\{_{f_t(p)}K_p\}_{p\in\mathbb{H}^2}$ of $G$ (given by Proposition \ref{Lip}) converge to the geodesic fibration $\{_{Y(p)}\mathfrak{k}_p\}_{p\in\mathbb{H}^2}$ of $\mathfrak{g}$ (given by Proposition \ref{lip}) as $t\rightarrow 0$, by scaling up $G$ near the identity at rate $1/t$ (to find $\mathfrak{g}$ in the limit).

We can interpret this as saying that the operation of passing to a rescaled limit (from $G$ to $\mathfrak{g}$, as in Figure \ref{AdSMink}) ``commutes'' with the operation of passing to a quotient (by $(\rho_t, j)(\Gamma)$ and, in the limit, $(u, j)(\Gamma)$). A more precise formulation can be given by finding \emph{sections} of the above geodesic fibrations: by definition, a section of the $f_t$-fibration $\pi_t:G\rightarrow \mathbb{H}^2$ of $G$  from Proposition \ref{Lip} is just a procedure taking each $p\in\mathbb{H}^2$ to an isometry $g^t_p\in G$ of $\mathbb{H}^2$ that sends $p$ to $f_t(p)$. Such a section is equivariant if $$g^t_{j(\gamma)\cdot p}=\rho(\gamma)g^t_pj(\gamma)^{-1} \hspace{10pt} \text{ for all } \hspace{10pt} p\in \mathbb{H}^2~.$$ Similarly, a section of the $Y$-fibration $\pi':\mathfrak{g}\rightarrow \mathbb{H}^2$ of $\mathfrak{g}$ from Proposition \ref{lip} is just a procedure taking each $p\in\mathbb{H}^2$ to a Killing field $X_p$ on $\mathbb{H}^2$ such that $Y(p)=X_p(p)$. Such a section is equivariant if $$X_{j(\gamma)\cdot p}=u(\gamma)+\mathrm{Ad}(j(\gamma))(X_p) \hspace{10pt} \text{ for all } \hspace{10pt} p\in \mathbb{H}^2~.$$ We say that the sections $(g^t_p)_{p\in\mathbb{H}^2}$ of $\pi_t$ converge to the section $(X_p)_{p\in\mathbb{H}^2}$ of $\pi'$ if $$\left . \frac{\mathrm{d}}{\mathrm{d}t}\right |_{t=0} g^t_p=X_p \hspace{10pt} \text{ for all } \hspace{10pt} p\in\mathbb{H}^2~.$$ Using Theorem \ref{integrate}, it is not hard to construct smooth sections with these properties, by choosing for instance isometries $g^t_p\in G$ osculating the smooth maps $f_t$ near $p$.

\begin{theorem} \cite[Cor.\ 1.5]{DGK1} \label{transition}
Under the hypotheses of Theorem \ref{integrate}, let $S$ denote the hyperbolic surface $j(\Gamma)\backslash \mathbb{H}^2$, and $\mathbb{S}^1$ denote the circle $\mathbb{R}/2\pi\mathbb{Z}$. The representations $(\rho_t, j)$ define complete Lorentzian structures $\omega_t$ on $M:=S\times \mathbb{S}^1$ modelled on $G$, and $(u, j)$ defines a complete Lorentzian structure $\omega$ on $M':=S\times (\mathbb{S}^1\smallsetminus \{[\pi]\})$ modelled on $\mathfrak{g}$, such that the scaled restrictions $t^{-2}\omega_t|_{M'}$ converge to $\omega$ as $t\rightarrow 0$.
\end{theorem}
By definition, a Lorentzian structure on a $3$-manifold $M$ is a smooth symmetric bilinear form on the tangent bundle $TM$ with signature $(-,+,+)$. We say that such a structure is \emph{modelled} on $G$ (which is itself a Lorentzian $3$-manifold) if it can be obtained by patching open sets of $G$ together via isometries of $G$. The structure is \emph{complete} if in addition it makes $M$ a quotient of the universal cover of $G$. (Similar definitions can be made with $\mathfrak{g}$ instead of $G$.) Lorentzian structures converge if the quadratic forms in each tangent space converge, uniformly on compact sets.

\begin{proof}
We sketch a proof of Theorem \ref{transition} (details can be found in \cite[\S 7]{DGK1}). For small $t>0$, place a Lorentzian form $\omega_t$ on $M:=S\times \mathbb{S}^1$, as follows. The maps $f_t$ produced by Theorem \ref{integrate} induce fibrations $\{_{f_t(p)}K_p\}_{p\in\mathbb{H}^2}$ of $G$ (Proposition \ref{Lip}) of which we consider equivariant, converging sections $(g^t_p)_{p\in\mathbb{H}^2}$ as explained before the statement of the theorem. These sections can be used to build \emph{diffeomorphisms} $$\Phi_t:M\rightarrow (\rho,j)(\Gamma)\backslash G~:$$ namely, the corresponding map $\widetilde{\Phi}_t$ between universal covers is defined so that $\widetilde{\Phi}_t(p,0)$ is the isometry $g^t_p$ of $\mathbb{H}^2$ (taking $p$ to $f_t(p)$), and $\widetilde{\Phi}_t(p,\theta)$ is the isometry $g^t_p$ followed by a rotation of angle $\psi_t(\theta)$ around $f_t( p)$, where $\psi_t$ is an appropriate self-diffeomorphism of $\mathbb{S}^1$ fixing $[0]$, \emph{of derivative $t$ at $[0]$}. Thus, $\widetilde{\Phi}_t(\{p\}\times \mathbb{S}^1)$ is the whole fiber ${}_{f_t(p)}K_p$ of isometries taking $p$ to $f_t(p)$, for each $p\in\mathbb{H}^2$. This construction descends to a quotient $\Phi_t$: by pushforward, $\Phi_t$ endows $(\rho,j)(\Gamma)\backslash G$ with a product structure compatible with the $f_t$-fibration into geodesics; by pullback, $\Phi_t$ endows the fixed manifold $M$ with a complete Lorentzian structure $\omega_t$.

A similar diffeomorphism $\Phi:M'\simeq S\times \mathbb{R}\rightarrow (u,j)(\Gamma)\backslash \mathfrak{g}$ can be built at the infinitesimal level, by defining the universal cover map $\widetilde{\Phi}$ to take each $(p,0)\in\mathbb{H}^2\times \mathbb{R}$ to the Killing field $X_p:=\left . \frac{\mathrm{d}}{\mathrm{d}t}\right |_{t=0}g^t_p \in \mathfrak{g}$, and $(p,\theta)$ to the sum of $X_p$ and an infinitesimal rotation of angular velocity $\theta$ around $p$. By pushforward, again $\Phi$ endows $(u,j)(\Gamma)\backslash \mathfrak{g}$ with a product structure compatible with the $Y$-fibration into geodesics (where $Y$ is the vector field given by Theorem \ref{integrate}); by pullback, $\Phi$ endows the fixed manifold $S\times \mathbb{R}\simeq M'$ with a complete Lorentzian structure $\omega$. The convergence statements of Theorem \ref{integrate} (and of sections) can be interpreted as saying that $t^{-2}\omega_t$ converges to $\omega$ (the exponent $-2$ appears because $\omega_t, \omega$ are $2$-forms). The reason why the second factor $\mathbb{S}^1$ becomes $\mathbb{R}\simeq \mathbb{S}^1\smallsetminus \{-1\}$ in the limit is that rescaling by $t^{-1}$ makes the fibers ${}_{f_t(p)}K_p$ appear longer and longer, with the limit being just a line.
\end{proof}

\subsection{A classification result} \label{classification}
In \cite{DGK2}, we prove a classification result for pairs $(u,j)$ that act properly dicontinuously on $\mathfrak{g}$, with $j$ convex-cocompact. The starting point is the following construction (due to Thurston \cite{stretchmap}; see also \cite{pt10}) to produce, from $j$, a representation $\rho$ of the free group $\Gamma$ such that 
there exist $(j,\rho)$-equivariant maps from $\mathbb{H}^2$ to itself with Lipschitz constant $<1$.
\begin{itemize}
\item Subdivide the hyperbolic surface $S=j(\Gamma)\backslash \mathbb{H}^2$ into disks, using disjoint, isotopically distinct, embedded geodesics $\alpha_1, \dots, \alpha_r$ with endpoints out in the funnels of $S$.
\item Near each $\alpha_i$, choose a second geodesic $\alpha'_i$ disjoint from $\alpha_i$. Let $(p_i, p'_i)\in \alpha_i\times \alpha'_i$ be the closest pair of points in $\alpha_i\times \alpha'_i$.
\item Delete from $S$ the $r$ strips bounded by $\alpha_i\cup \alpha'_i$ (which are assumed disjoint), and glue back $\alpha_i$ to $\alpha'_i$ by identifying $p_i$ with $p'_i$. The representation $\rho$ is the holonomy of the new hyperbolic metric, defined up to conjugacy. 
\end{itemize}

Foliating the strip between $\alpha_i$ and $\alpha'_i$ by geodesics perpendicular to $[p_i, p'_i]$, all collapsed to a single line in the new metric, we see that there exists a $1$-Lipschitz, $(j,\rho)$-equivariant map $f$. This constant $1$ can be slightly improved by an appropriate ``relaxation'' procedure. We say that $\rho$ is a \emph{strip deformation} of $j$.

\begin{definition}
If $(u,j):\Gamma\rightarrow \mathfrak{g}\rtimes G = \mathbf{G}'$ is a representation, with $j$ is Fuchsian, such that there exist $(j,u)$-equivariant vector fields on $\mathbb{H}^2$ with negative lipschitz constant, then we call $u$ an \emph{admissible} cocycle.
\end{definition}
The infinitesimal version of the above construction, obtained as $d(p_i, p'_i)$ goes to~$0$, produces an admissible $j$-cocycle $u:\Gamma\rightarrow \mathfrak{g}$. By definition, the \emph{support} of an infinitesimal strip deformation is the union of the arcs~$\alpha_i$.
\begin{theorem} \cite[Th.\ 1.4]{DGK2} \label{triangulation}
If $j$ is convex cocompact, then all admissible cocycles $u$ are realized as infinitesimal strip deformations. Moreover, the realization is unique if we request that $\alpha_i$ exit the convex core $\Omega$ of $S$ perpendicularly at both ends and that $p_i$ be the midpoint of $\Omega\cap \alpha_i$.
\end{theorem}
\begin{proof}
To sketch the main ideas of the proof (see \cite{DGK2} for details), we recall the \emph{arc complex} $\overline{\mathcal{X}}$ of $S$: this is the abstract simplicial complex with one vertex for each embedded line in $S$ crossing $\partial\Omega$ perpendicularly (such as the $\alpha_i$), and one $k$-dimensional cell for each system of $k+1$ disjoint such lines. The combinatorial complex $\overline{\mathcal{X}}$ does not depend on the hyperbolic metric on $S$. Next, define $\mathcal{X}\subset\overline{\mathcal{X}}$ as the complement in $\overline{\mathcal{X}}$ of the union of all simplices corresponding to systems of arcs that \emph{fail} to decompose $S$ into topological disks. It is a theorem of Penner \cite{penner} that $\mathcal{X}$ is locally finite, and homeomorphic to an open ball of dimension one less than the Teichm\"uller space of~$S$. The construction of strip deformations gives a natural map $$f : \mathcal{X}\rightarrow \mathbb{P}(\mathsf{adm})$$ where $\mathsf{adm}$ denotes the space of admissible cocycles. Namely, if $x=(t_i, \alpha_i)_{1\leq i \leq r}$ is a point of a simplex of $\mathcal{X}$, with $t_i>0$ and $\sum_i t_i =1$, then the associated (projectivized) cocycle $f(x)$ is the one corresponding to infinitesimal strip deformation along the $\alpha_i$, with the weight $t_i$ being the infinitesimal width, at its midpoint $p_i$, of the collapsing strip supported by $\alpha_i$. It is enough to prove that $f$ is a homeomorphism. 
We do this in several steps:

First, admissible cocycles form a convex cone in the tangent space to the space of representations: this can be seen from Theorem \ref{mainmicro} by interpolating linearly between contracting vector fields (if the $Y_i$ are $(j,u_i)$-equivariant, $c_i$-lipschitz vector fields, then $\sum \lambda_i Y_i$ is $(j,\sum \lambda_i u_i)$-equivariant and $(\sum \lambda_i c_i)$-lipschitz). The same line of argument shows that this cone is open, of full dimension, in the space of infinitesimal deformations of~$j$. (Under a different, but equivalent definition of admissibility, this important result was first proved in \cite{GLM}.) Therefore, the range of $f$ is an open ball, whose dimension turns out to be the same as the domain. It is then enough to prove that $f$ is a covering: namely, a proper map that is locally a homeomorphism.

Second, $f$ is a proper map \cite[\S 3.1]{DGK2}. This means that if $(x_n)_{n\in\mathbb{N}}$ goes to infinity in $\mathcal{X}$, then the limit $[u]$ of the (projectivized) cocycles $f(x_n)$ is not admissible. If the supports of the $x_n$ stabilize in the arc complex, say to a decomposition $\Delta$ of $S$ into disks, then $x_n\rightarrow\infty$ means that too many edges of $\Delta$ have weights going to~$0$, in the sense that $S$ contains a closed geodesic, representing some $\gamma\in \Gamma$, that intersects only edges of just that type. This means that, without loss of generality, $u(\gamma)=0$, a clear obstruction to finding $(j,u)$-equivariant, contracting vector fields on $\mathbb{H}^2$ (any such field would have to be periodic on the axis of $\gamma$). If the supports of the $x_n\in \mathcal{X}$ do not stabilize, then they consist of longer and longer arcs, limiting (in the Hausdorff sense) to some geodesic lamination $\Lambda$ not all of whose leaves exit~$S$. We can then apply a similar argument, replacing $\gamma$ with curves nearly carried by~$\Lambda$.

Third, $f$ is locally a homeomorphism. We prove this at each point $x\in \mathcal{X}$, depending on $d\geq 0$, the codimension in $\mathcal{X}$ of the smallest cell (or \emph{stratum}) of $\mathcal{X}$ that contains~$x$. If $d\in \{0,1,2\}$, this can be proved directly in terms of elementary inequalities in Lorentzian geometry \cite[\S 5]{DGK2}. For $d\geq 3$, we can use an inductive argument: the link of the map $f$ at the stratum of $x$ is a piecewise projective map from a (simplicial) $(d-1)$-sphere to a $(d-1)$-sphere, and is a local homeomorphism by induction, therefore a covering: but when $d\geq 3$ the $(d-1)$-sphere can only cover itself trivially. Hence, $f$ is a local homeomorphism in a neighborhood of the stratum of $x$.
\end{proof}

Theorem \ref{triangulation} has analogues where the arcs $\alpha_i$ do not necessarily exit the convex core at right angles, and where the points $p_i\in \alpha_i$ are arbitrary. It also has macroscopic analogues (with representations $\rho:\Gamma\rightarrow G$ instead of cocycles $u:\Gamma\rightarrow \mathfrak{g}$). Moreover, the construction of strip deformations produces naturally a (discontinuous, piecewise Killing) equivariant vector field with lipschitz constant~$0$. An analogue, or limit case, of Proposition \ref{lip} can be proposed for such vector fields, producing ``degenerate'' fibrations of $\mathfrak{g}$ into straight lines (some of which may intersect, but ``benignly''). These degenerate fibrations can be conveniently described in terms of \emph{crooked planes}, which are $\mathrm{PL}$ surfaces of $\mathfrak{g}$ introduced by Drumm in \cite{dru92} and studied since by several authors \cite{dg95, dg99, cg00, cdg10, cdg11, gol13, bcdg13, cdg13}. Thus, a corollary of Theorem \ref{triangulation} is

\begin{corollary} \cite[Th.\ 1.6]{DGK2}
If $\Gamma$ is a free group, $j:\Gamma\rightarrow G$ is convex cocompact and $u:\Gamma\rightarrow \mathfrak{g}$ an admissible cocycle, then the complete Lorentz manifold $(u,j)(\Gamma)\backslash \mathfrak{g}$ admits a fundamental domain bounded by crooked planes.
\end{corollary}

This had been conjectured in \cite{dg95}. In \cite[\S 8]{DGK2} we also give an analogue for quotients of $G$.

\end{document}